\documentclass[11pt]{amsart}
\usepackage{graphicx,color}
\usepackage{bm}% bold math
\usepackage{amsmath}
\usepackage{amstext}
\usepackage{amsfonts}
\usepackage{amssymb}
\usepackage{amsthm}
\usepackage{amsrefs}
\usepackage{url}
\usepackage{hyperref}
\hypersetup{
    colorlinks=true,
    linkcolor=blue,
    citecolor=blue,      
    urlcolor=blue,
}

\newcommand{\bb}[1]{\mathbb{#1}}
\newcommand{\cl}[1]{\mathcal{#1}}

\newcommand{\R}{\mathbb{R}}
\newcommand{\C}{\mathbb{C}}
\newcommand{\F}{\mathbb{F}}

\newcommand{\spc}[1]{\mathcal{#1}}

\newcommand{\Span}{{\operatorname{span}}}

\def\>{\rangle}
\def\<{\langle}

\newcommand{\map}[1]{\mathcal{#1}}

\newcommand{\id}{{\operatorname{id}}}

\newcommand{\Sym}{{\operatorname{Sym}}}

\newcommand{\AND}{\text{ and }}

\newcommand{\hS}{{\widehat {\map S} }}
\newenvironment{spmatrix}{\left(\begin{smallmatrix}}{\end{smallmatrix}\right)}

\theoremstyle{plain}
\newtheorem{thm}{Theorem}[section]
\newtheorem{theo}[thm]{Theorem}

\newtheorem{cor}[thm]{Corollary}
\newtheorem{prop}[thm]{Proposition}

\newtheorem{lemma}[thm]{Lemma}

\theoremstyle{definition}

\usepackage{qcircuit}

\begin{document}
%%%%%%%%%%%%%%%%%%%%%%%

\title[Counterexamples to extendibility]{Counterexamples to the extendibility\\ of positive  unital norm-one maps}

\author[G.~Chiribella]{Giulio Chiribella} 
\address{Department of Computer Science, The University of Hong Kong, Pok Fu Lam Road, Hong Kong  and  
Department of Computer Science, University of Oxford, Wolfson Building, Parks Road, Oxford, UK}
\email{giulio.chiribella@cs.ox.ac.uk}

\author[K.R.~Davidson]{Kenneth R. Davidson}
\address{Department of Pure Mathematics, University of Waterloo, Waterloo, ON, Canada N2L 3G1}
\email{krdavidson@uwaterloo.ca}

\author[V.I.~Paulsen]{Vern I.~Paulsen}
\address{Institute for Quantum Computing and Department of Pure Mathematics, University of Waterloo,
Waterloo, Waterloo, ON, Canada N2L 3G1}
\email{vpaulsen@uwaterloo.ca}

\author[M.~Rahaman]{Mizanur Rahaman}
\address{Univ Lyon, ENS Lyon, UCBL, CNRS, Inria, LIP, F-69342, Lyon Cedex 07, France}
\email{mizanur.rahaman@ens-lyon.fr}

%%%%%%%%%%%%%%%%%%%%%%%
 \begin{abstract}
 Arveson's extension theorem guarantees that every completely positive map defined on an operator system  
 can  be extended to a completely positive map defined on the whole C*-algebra containing it. 
 An analogous statement where complete positivity is replaced by positivity is known to be false.   
 A natural question is whether  extendibility could still hold for positive maps satisfying stronger conditions, 
 such as being unital and norm 1.    
 Here we provide three counterexamples showing that  positive norm-one unital maps defined on an operator subsystem 
 of a matrix algebra cannot be extended to a positive map on the full matrix algebra.  
 The first counterexample is an unextendible positive unital map  with unit norm, 
 the second counterexample is an unextendible positive unital isometry on a real operator space,
 and the third counterexample is an unextendible positive unital isometry on a complex operator space.  
 %Although the third example is arguably superior to the first example, the relative dimension of the first example is much larger than that of the third example.
\end{abstract}

\maketitle

%%%%%%%%%%%%%%%%%%%%%%%
\section{Introduction}
%%%%%%%%%%%%%%%%%%%%%%%

Arveson's Extension Theorem \cite{arveson69} is a central result in the theory of operator algebras. 
It states that if $\mathcal{A}$ is a C$^*$-algebra,  $\mathcal{S}$ an operator system contained in $\mathcal{A}$, 
and $\Phi:\mathcal{S}\rightarrow B(\mathcal{H})$ is a completely positive map from $\mathcal{S}$ to the 
bounded linear operators on a Hilbert space $\mathcal{H}$, then there exists a completely positive map 
$\Psi:\mathcal{A}\rightarrow B(\mathcal{H})$ extending the map $\Phi$ to the whole algebra $\cl A$. 
A natural question is whether some analogue of Arveson's theorem holds for maps that are positive but not completely positive.   
The most immediate analogue fails to hold:  already in \cite{arveson69}, Arveson provided a  counterexample 
showing that the mere positivity of $\Phi$ is not sufficient for it to have a positive extension. 
His example (see also \cite[Examples 2.2, 2.13]{paulsenbook}) is the map on $\cl S = \Span\{ 1, z, \bar z \} \subset C(\bb T)$,
where $C(\bb T)$ is the C$^*$-algebra of continuous functions on the unit circle $\bb T$, $z$ is the coordinate function, and $\Phi:\mathcal{S}\rightarrow M_2(\C)$ is given by 
\[
 \Phi(a+bz +c\bar{z}) =  \begin{pmatrix}  a  & 2b  \\  2c  &  a   \end{pmatrix}  . 
\]
In this counterexample, the obstruction to extendibility of the map $\Phi$ appears to be its norm.  
A result by  Russo and Dye  (Corollary 2.9 in \cite{paulsenbook}) guarantees that a positive map on a unital C$^*$-algebra attains its norm at the identity. 
Hence, every extendible unital map should have norm $1$. 
The map $\Phi$ is unital and has norm $2$, and therefore it cannot be extended.

Arveson's counterexample stimulates the question whether every unital positive map with unit norm admits a positive extension.  
In this article we present three counterexamples to the above claim.  
Our first counterexample is a positive unital map $\Phi_n$, defined on an operator system of $M_{2n}  (\C)$, with the property that $\|  \Phi_n\|=1$ 
but $\Phi_n$ does not admit any positive extension to $M_{2n}  (\C)$ for every $n  > 16$.  
This counterexample applies to an early statement, made by St\o rmer in \cite[Theorem 4] {stormerarxiv}, 
and later amended in the published version \cite{stormer-pub}.   
This counterexample was cited as the third reference of  the published version \cite{stormer-pub}, but was not publicly available previously.

Our second counterexample is a positive isometry $\Upsilon_n$  defined on a self-adjoint subspace $\spc S_n$ containing the identity of $M_{2n}  (\R)$, 
with the property that $\Upsilon_n$ is unextendible for every $n\ge 2$.  This example should be compared with \cite[Theorem 1]{stormer-pub}, which provides a sufficient condition for extendibility in terms of a notion of  $\spc C$-positivity,  $\spc C$ being a cone of linear maps.  Specifically, the theorem states that every $\spc C$-positive map on a real operator system can be extended to a $\spc C$-positive map on the whole algebra.
When $\spc C$ is the cone of all positive maps,  $\spc C$-positivity of a map defined on the whole algebra is equivalent to positivity.   
Our  example   shows that  $\spc C$-positivity of a map defined on  the real operator system $\spc S_n$ 
is a strictly stronger property than positivity of the map on $\spc S_n$.  

Finally, our third counterexample is a positive isometry $\Gamma_n$  defined on an operator system of $M_{2n}  (\C)$, 
with the property that $\Gamma_n$ is unextendible for every $n\ge 2$. 
%Moreover, $\Gamma_n$ is the complexification of a real map $\Gamma_n$ on a real operator system.
%We show that $\Gamma_n$ also does not have a positive extension to $M_{2n}(\R)$.

Although our third counterexample is in some sense stronger than the first
%two, 
we present the first for historical reasons.  
The second is used to establish parts of the third example.
Also, the relative dimensions of the operator systems in the first and third examples are quite different. 
 In the first example, the dimension of the operator system is roughly 1/2 the dimension of the containing matrix algebra, 
while  in the third counterexample  the dimension of the operator system is roughly 1/4 the dimension of the containing matrix algebra.

An interesting question for future research is whether there are bounds on how large  the fractional dimension of an operator system must be in order for 
non-extendible positive maps of the types that we construct to exist.

%%%%%%%%%%%%%%%%%%%%%%%
\section{An unextendible positive unital map with unit norm} 
%%%%%%%%%%%%%%%%%%%%%%%

%%%%%%%%%%%%%%%%%%%%%%%
\begin{theo}\label{davidson-paulsen ex}
Consider the operator system 
\begin{align*} 
\spc A_n :  = \left\{ \begin{pmatrix}  a  I_n   & B \\    C &  d  I_n  \end{pmatrix} : a,d \in \C ,\  B,C  \in M_n  (  \C)  \right\} \subseteq M_2(M_n(\C))  , 
\end{align*}
and define a map $\Phi_n: \spc A_n\rightarrow M_{2n}(\C)$ by 
\begin{align*}
  \Phi_n   \begin{pmatrix}   a  I_n   & B  \\   C &  d  I_n        \end{pmatrix}  
  =  \begin{pmatrix}   a I_n   & \frac{1}{4}{B}^{t}  \\   \frac{1}{4} {C}^{t}  &  d I_n  \end{pmatrix} , 
\end{align*}
where $B^t$ denotes the transpose of B. 
Then $\Phi_n$ is a unital, positive map with \mbox{$\|\Phi_n\|=1$}, 
that does not admit a positive extension to $M_{2n}(\C)$ for $n>16$.
\end{theo}

\begin{proof}
It is easy to see that $\Phi_n$ is a positive map. For  $n\le 4$, the map is also completely positive.  Indeed,  it is well known that the map $B \to \frac14 B^t$ is a completely contractive map on $M_n (\C)$ for $n  \le 4$, and   thus  \cite[Lemma 8.1]{paulsenbook} implies that $\Phi_n$ is completely positive for   $n\le 4$.  

We now show that $\|\Phi_n\|  =1$ for every $n$.   For $n\le 4$, this follows from the fact that $\Phi_n$ is completely positive, and therefore extendible to a completely positive map on the whole algebra. Hence, the Russo-Dye result (Corollary 2.9 in \cite{paulsenbook})  implies $\| \Phi_n\|   =  \|  \Phi (I_{2n})\|=1$.  
For $n>4$, consider a generic matrix $ R = \begin{spmatrix} aI_n & B \\ C & dI_n \end{spmatrix} \in \spc A_n$  and write 
\[
 \left\| \Phi_n(R) \right\| =  \Big\langle \begin{pmatrix} aI_n & \frac14 B^t \\ \frac14 C^t & dI_n \end{pmatrix} 
 \begin{pmatrix} x_1 \\ x_2 \end{pmatrix} \Big\vert \begin{pmatrix} y_1 \\ y_2 \end{pmatrix} \Big\rangle.
\]
for some suitable vectors  $x_1, x_2, y_1, y_2$  in $\C^n$  satisfying the normalization condition $\|x_1\|^2 + \|x_2\|^2 = \|y_1 \|^2 + \|y_2\|^2 =1$.

Let $k\le 4$ be the dimension of the  subspace $ \spc S  :  = \Span \{x_1,x_2,y_1,y_2\}$ and  let   $V \in  M_{n,k}   (\C)$ be an  isometry satisfying $  V  \C^k  = \spc S$. 
 Then, 
\begin{align*}
\Big\langle \begin{pmatrix} aI_n & \frac14 B^t \\ \frac14 C^t & dI_n \end{pmatrix} 
 \begin{pmatrix} x_1 \\ x_2 \end{pmatrix} \Big\vert \begin{pmatrix} y_1 \\ y_2 \end{pmatrix} \Big\rangle    \le  \left\|  \begin{pmatrix} aI_k & \frac14   V^* B^t  V \\ \frac14   V^* C^t   V& d I_k \end{pmatrix}      \right\|=  \left\|    \Phi_k (   R') \right\|  \, ,
\end{align*}
with $R':  =\begin{spmatrix}       a I_k &  V^t  B  V^{  \ast  t}       \\  V^t C   V^{\ast  t}& d I_k   \end{spmatrix}$.  Since $\|  \Phi_k\|=1$,  $\|  \Phi_k  (R') \|  \le  \|  R'\|$.   Moreover, we have    
\[   \|  R'\|     =   \left\|  \begin{pmatrix}       V^{\ast t}  &    0        \\    0&   V^{\ast t}  \end{pmatrix}   \,  R'  \,  \begin{pmatrix}       V^{t}  &    0        \\    0&   V^{t}  \end{pmatrix}   \right\|   =  \\   \left\|  \begin{pmatrix}       a \,  P  &    P B   P        \\    P C   P& d \,  P    \end{pmatrix} \right\|   \le \|  R\|   \,, \]
having defined the projection $P:  =  V^{  \ast  t}  V^t$.   In summary, we have shown that $\|  \Phi_n (R)\|  \le \|  R\|$ for every matrix $R\in  \spc A_n$. Since $\Phi_n$ is unital, $\| \Phi_n\| = 1$.

Now assume that $\Phi_n$ could be extended to a positive map $\Psi_n$ on $M_{2n}  (\C)$. 
Let $E_{i,j}$ denote the usual matrix units,  with a 1 in the $ij$-entry, and zeros everywhere else.
Note that if $0 \le P \le I_n$, then
\[
 0 \le \Psi_n \big( \begin{pmatrix} P & 0 \\0 & 0 \end{pmatrix} \big) 
 \le \Psi_n \big( \begin{pmatrix} I_n &0 \\0 &0 \end{pmatrix} \big) 
 = \begin{pmatrix} I_n & 0 \\ 0 & 0 \end{pmatrix}.
\]
This implies that for any $X \in M_n$, there is some $Y \in M_n$ so that
\[
 \Psi_n \big( \begin{pmatrix} X & 0 \\ 0 & 0 \end{pmatrix} \big)  
 = \begin{pmatrix} Y & 0 \\ 0 & 0 \end{pmatrix}.
\]

Let
\[
 \Psi_n \big( \begin{pmatrix} E_{i,i} & 0 \\ 0 & 0 \end{pmatrix} \big) = \begin{pmatrix} P_i & 0 \\ 0 & 0 \end{pmatrix}.
\] 
Since $\sum_{i=1}^n E_{i,i} = I_n$, we have that $\sum_{i=1}^n P_i = I_n$.

Since for each $1 \le i,j \le n$,
\[
 Q_{ij} :=  \begin{pmatrix} E_{i,i} & E_{i,j} \\ E_{j,i} & I_n \end{pmatrix} \ge 0,
\]
we would have that
\[
0 \le \Psi_n(Q_{ij}) = \begin{pmatrix} P_i & \frac14 E_{j,i} \\ \frac14 E_{i,j} & I_n \end{pmatrix}.
\]
From this it follows that $P_i \ge \frac1{16} E_{jj}$ for $1 \le i \le n$.
Therefore
\[
 I_n = \sum_{i=1}^n P_i \ge \frac{n}{16} E_{j,j}.
\]
This is a contradiction when $n>16$.

Therefore, no positive extension can exist when $n>16$.
\end{proof}

%%%%%%%%%%%%%%%%%%%%%%%

%%%%%%%%%%%%%%%%%%%%%%%
\section{An unextendible positive isometry, real case} 
%%%%%%%%%%%%%%%%%%%%%%%

In this section we present an example of a positive isometric map on a real operator subsystem of a real matrix algebra that is not extendible to a positive map on the full matrix algebra. What it means for an operator on a real Hilbert space $\cl H$ to be positive is generally accepted:
\[
 P \ge 0 \iff P = P^* \AND \langle Px, x \rangle \ge 0, \, \forall x \in \cl H \iff \exists X,\ P = X^*X  .
\]

However, what is meant by a real operator system and by a map between real operator systems to be positive is a bit ambiguous. 
See \cite{BlecherTepsan} and the references therein for further discussion of what are the appropriate definitions of operator systems, positive maps and completely positive maps over real Hilbert spaces.  
Also see \cite{Ruan} for proofs of, appropriately defined, real versions of many classical results, such as Arveson's Extension Theorem and Stinesping's Representation Theorem.  An investigation on the relation between positivity on the real field and positivity on the complex field is provided  in  a forthcoming paper \cite{CDPR}

Here we adopt the terminology that a map $\Phi$ between spaces of real matrices is {\em positivity preserving} provided that
$ P \ge 0 \implies \Phi(P) \ge 0.$
Note that for complex operator systems this corresponds to our usual definition of a complex linear map being positive.  
We will call $\Phi$ {\em positive} if it is positivity preserving and {\em self-adjoint}, i.e., $\Phi(A^t) = \Phi(A)^t$.

In what follows we use $\F$ to denote either $\R$ or $\C$.

We need the following immediate consequence of \cite[Lemma 3.1]{paulsenbook}.

%%%%%%%%%%%%%%%%%%%%%%%
\begin{lemma}\label{lem:Snpositive}\ 
\begin{enumerate}
\item A matrix $M = \begin{spmatrix}   a  I_n   & C  \\  C^* &  b  I_n   \end{spmatrix} \in M_2(M_n(\F))$ \vspace{.3ex}
is positive if and only if  $a\ge 0$,  $b\ge 0$, and $\|  C \|  \le \sqrt{ab}$.  \vspace{.3ex}

\item A matrix $N = \begin{spmatrix}   A   & b I_n  \\  c I_n &  d  I_n   \end{spmatrix} \in M_2(M_n(\F))$ \vspace{.3ex} 
is positive if and only if $A\ge 0$, $c=\bar b$, $d\ge0$ and $dA \ge |b|^2 I_n$.
\end{enumerate}
\end{lemma}

\begin{proof}
(1) Clearly $M \ge 0$ requires $a\ge0$ and $b\ge0$. If $ab=0$, then we also need $C=0$.
So we may suppose that $ab>0$.
Factor $M$ as
\[
 M = \begin{pmatrix}   a  I_n   & C  \\  C^* &  b  I_n   \end{pmatrix}
 = \begin{pmatrix}   \sqrt a\,  I_n   & 0  \\  0 &  \sqrt b  \,I_n   \end{pmatrix}
 \begin{pmatrix}     I_n   & \frac 1 {\sqrt{ab}} C  \\  \frac 1 {\sqrt{ab}} C^*  &    I_n   \end{pmatrix}
 \begin{pmatrix}   \sqrt a  \,I_n   & 0  \\  0 &  \sqrt b  \,I_n   \end{pmatrix} .
\]
Then $M\ge0$ if and only if the middle factor is positive.
By \cite[Lemma 3.1]{paulsenbook}, this holds precisely when $\|C/\sqrt{ab} \| \le 1$.

(2) Similarly, for $N\ge0$, we require $c=\bar b$ and $d\ge0$.
If $d=0$, then $b=c=0$ and $A\ge0$ is sufficient.
If $d>0$, factor
\[
 N = \begin{pmatrix}   A   & b I_n  \\  c I_n &  d  I_n   \end{pmatrix}
 = \begin{pmatrix}   I_n   & 0  \\  0 &  \sqrt d\,  I_n   \end{pmatrix}
 \begin{pmatrix}   A   & \frac b {\sqrt d}\, I_n  \\  \frac {\bar b} {\sqrt d}\, I_n &   I_n   \end{pmatrix}
 \begin{pmatrix}   I_n   & 0  \\  0 &  \sqrt d\,  I_n   \end{pmatrix} .
\]
This is positive if and only if the middle factor is positive.
By \cite[Lemma 3.1]{paulsenbook}, this holds precisely when $A \ge \frac{|b|^2}d I_n$.
\end{proof}

We set
\begin{align}\label{Sn}
 \spc S_n :  =  \left\{  \begin{pmatrix}   a  I_n   & C  \\  C^t &  b  I_n   \end{pmatrix} :  a\in \R , \ b \in  \R, \ C  \in M_n  (  \R)  \right\}  \subseteq M_2  ( M_n(\R) ) .
\end{align} 
Then $\spc S_n$ is contained in $\Sym_{2n}(\R)$ and it is spanned by its positive elements.

%%%%%%%%%%%%%%%%%%%%%%%
\begin{theo}\label{Phinunextendible}
Consider the map  $\Upsilon_n:  \spc S_n \to \spc S_n$ defined by 
\begin{align} \label{Phin}
  \Upsilon_n  \begin{pmatrix}   a  I_n   & C  \\  C^t  &  b  I_n        \end{pmatrix}   
  =  \begin{pmatrix}   a  I_n   & C^t  \\    C  &  b  I_n        \end{pmatrix}
\end{align}
The map $\Upsilon_n$ is unital, is an involution, i.e.,  $\Upsilon_n^2 =  \id_{\cl S_n}$, and
\begin{enumerate}
 \item   $\Upsilon_n$ is positivity preserving,
 \item $\Upsilon_n$ is self-adjoint,  \item   $\| \Upsilon_n \| = 1$, and   
 \item   $\Upsilon_n$ does not admit a positive preserving extension to $M_{2n}(\R)$ for any $n\ge 2$.   
\end{enumerate}
\end{theo}

\begin{proof}
 It is immediate from the definition that $\Upsilon_n$ is unital, is an involution, and self-adjoint.  We now  show that $\Upsilon_n$ is positivity preserving.
 
Consider a positive element $M =   \begin{pmatrix}   a  I_n   & C  \\  C^t  &  b I_n  \end{pmatrix}   \in \spc S_n$.
By Lemma~\ref{lem:Snpositive}, $a,b$ are positive and $\|  C \|  \le \sqrt{ab}$. Now
\[
 \Upsilon_n(M)  =  \begin{pmatrix}   a I_n   & C^t  \\  C &  b I_n   \end{pmatrix} .
\]
By Lemma~\ref{lem:Snpositive}, this is positive if $\| C^t \| = \|C\|  \le \sqrt{ab}$; which is true.    
Hence $\Upsilon_n$ is a positive preserving map.  

We next show that the map $\Upsilon_n$ has unit norm.  
This follows since the real operator space $\spc S_n$ consists of symmetric real matrices.  Every symmetric real matrix $B$ satisfies $-\|B\| I \le B \le \|B\| I$, and therefore
\[
 -\|B\| \|\Upsilon_n(I)\| \le \|\Upsilon_n(B)\| \le \|B\| \|\Upsilon_n(I)\| .  
\]
Hence $\| \Upsilon_n(B) \| \le \|\Upsilon_n(I)\| \|B\|$. So $\Upsilon_n$ has unit norm.

Following the proof of the Theorem \ref{davidson-paulsen ex}, it is also easy to see that $\Upsilon_n$ has 
no positive preserving extension to the whole $M_{2n}  (\R)$, except in the trivial case where $n=1$. 
The proof is by contradiction. 
We assume  that a positivity preserving extension  $\Psi_n:     M_{2n}  (\R)  \to M_{2n}  (\R )$ exists, 
and prove that its existence leads to a contradiction.  
   
Let $E_{ij} \in  M_n(\R)$ be the usual matrix units.  
Since $\Psi_n$ is positivity preserving, we must have  
\begin{align*}
 \Psi_n  \begin{pmatrix}  E_{ii} & 0 \\  0  &  0   \end{pmatrix}   
 &\le  \Psi_n  \begin{pmatrix}    I_n  & 0  \\  0  &  0  \end{pmatrix} % \\&
 =  \Upsilon_n \begin{pmatrix}  I_n  & 0 \\  0  &  0   \end{pmatrix}  %\\&
 =  \begin{pmatrix}    I_n    & 0  \\  0  &  0  \end{pmatrix}  .
\end{align*}
Equivalently, we must have 
\begin{align*}
 \Psi_n  \begin{pmatrix}  E_{ii} & 0 \\ 0  &  0  \end{pmatrix} 
 =    \begin{pmatrix}  P_{i}  & 0 \\ 0  &  0   \end{pmatrix} ,
\end{align*}
for some matrix $0\le P_i\le I_n$.  
Now, the matrix $M: =  \begin{pmatrix}  E_{ii}  & E_{ij}  \\  E_{ji}  &  I_n   \end{pmatrix}$ is positive, and we have  
\begin{align*}
 \nonumber \Psi_n ( M )    &  =  
 \Psi_n  \begin{pmatrix}  E_{ii}  & 0 \\  0  &  0   \end{pmatrix} + \Psi_n \begin{pmatrix}  0  & E_{ij}  \\  E_{ji}  &  I_n   \end{pmatrix}  \\     
 \nonumber    & =   \begin{pmatrix}  P_{i}  & 0  \\   0  &  0    \end{pmatrix}  
  + \Upsilon_n  \begin{pmatrix}  0  & E_{ij}  \\ E_{ji}  &  I_n   \end{pmatrix}  \\&  
 %\nonumber     
  =  \begin{pmatrix}  P_{i}   & 0  \\   0  &  0    \end{pmatrix} 
  +  \begin{pmatrix}    0  & E_{ji}  \\  E_{ij}  &  I_n    \end{pmatrix}  %\\&     
  =  \begin{pmatrix}  P_{i}  & E_{ji}  \\  E_{ij}  &  I_n  \end{pmatrix}  .
\end{align*}   
Since $M$ is positive and $\Psi_n$ is a positivity preserving, $\Psi_n(M)$ must be positive. 
By \cite[Lemma~3.1]{paulsenbook},
\[
 P_i \ge E_{ji}E_{ij} = E_{jj} \quad\text{for }1 \le i \le n.
\]
Therefore, $I_n =\sum_{i=1}^n P_i  \ge  n  E_{jj}$, a contradiction.   
So a positivity preserving extension $\Psi_n$ does not exist for $n\ge2$. 
\end{proof}

    Interestingly, the complexification of $\Upsilon_n$ is not an isometry.   
 For a map $\Phi$ on a real subspace $ \spc S$ of  $M_n(\R)$, its {\it complexification} 
is the map $\Phi'$ on 
\[ \spc S':= \spc S+ i \spc S \subseteq M_n(\C)\]
given by $$\Phi'(A+iB)=\Phi(A)+i\Phi(B),$$ where  $A,B\in \spc S$. 

Let $\spc S'_n$ be the complex subspace spanned by the real operator system $\spc S_n$, namely
\begin{align*} %\label{SnC}
 \spc S'_n :=  \left\{ \begin{pmatrix}   a  I_n   & C  \\ C^t  &  b  I_n   \end{pmatrix} :  a, b \in  \C, \ C  \in M_n  (  \C)  \right\}   \subseteq   M_2  \left(  M_n  (\C)\right) . 
\end{align*} 
Note that the positive elements in $\spc S'_n$ are exactly the same as the positive elements in $\spc S_n$: 
they are matrices with real entries satisfying $a\ge 0$, $b\ge 0$, and $\|  C\|  \le \sqrt{ab}$. 

Let  $\Upsilon'_n:  \spc S'_n \to \spc S'_n$ be the complexification of the map $\Upsilon_n$, namely 
\begin{align*} %\label{PhinC}
  \Upsilon'_n  \begin{pmatrix}   a  I_n   & C  \\ C^t  &  b  I_n  \end{pmatrix}
  =  \begin{pmatrix}   a  I_n   & C^t  \\  C  &  b  I_n  \end{pmatrix}  .
\end{align*}

 %%%%%%%%%%%%%%%%%%%%%%%
\begin{prop}
The map $\Upsilon'_n$ is positive and unital, but $\| \Upsilon'_n \| = \frac 2 {\sqrt 3}$.
\end{prop}

\begin{proof}
Since $\spc S'_n$ has the same positive elements as $\spc S_n$, it is clear that $\Upsilon'_n$ is positive.

For the norm, we first consider the $n=2$ case.  
Let
\begin{align*}\label{largernorm}
M = \left(   \begin{array}{cc|cc}     
1& 0 & 1&  0  \\ 
0 & 1 & i & 0 \\
\hline
1 & i & 0 &  0\\
0 & 0 & 0 & 0   \end{array} \right)
%\end{align*}
\quad\text{and}\quad
%\begin{align*}
M^* M = \left(   \begin{array}{cc|cc}     
2& -i & 1&  0  \\ 
i & 2 & -i & 0 \\
\hline
1 & i & 2 &  0\\
0 & 0 & 0 & 0   \end{array} \right) .
\end{align*}
The eigenvalues of $M^* M$ are $+3$ and $0$, both with multiplicity 2.  
Hence, $\| M \| = \sqrt 3$. 
On the other hand,  let 
\begin{align*}
 N:  = \Upsilon'_2  (M )= \left(   \begin{array}{cc|cc}     
 1& 0 & 1&  i  \\ 
 0 & 1 & 0 & 0 \\
\hline
1 & 0 & 0 &  0\\
i & 0 & 0 & 0   \end{array} \right) 
%\end{align*}
\quad\text{and}\quad 
%\begin{align*}
N^* N = \left(   \begin{array}{cc|cc}     
3& 0 & 1&  -i  \\ 
0 & 1 & 0 & 0 \\
\hline
1 & 0 & 1 &  -i\\
i & 0 & i & 1  \end{array} \right) .
\end{align*}
The eigenvalues of $N^* N$ are $+4$, $0$, and $+1$ with multiplicity 2. 
Hence, the norm of $\| N \| = 2 $, and thus $\| \Upsilon'_2 \| \ge 2/\sqrt 3$.   
Extending this example to $n\ge 2$, we obtain $\| \Upsilon'_n \| \ge 2/\sqrt 3$ for $n\ge 2$.  

We now show that the norm is exactly $2/\sqrt 3$.  
Let 
\begin{align*}
 M  = \begin{pmatrix}   a  I_n   & C  \\  C^t  &  b  I_n   \end{pmatrix}
\end{align*}
be a generic element of $\spc S_n'$ with $\| M\|  =1$.  The condition $\| M\|  =1$  implies 
\begin{align*}
 |a|^2  +  \|  C\|^2  \le 1  \qquad{\rm and} \qquad 
 |b|^2  +  \|  C\|^2  \le 1  .
\end{align*}
Without loss of generality, let us assume that $|b|  \ge |a|$. 

Now, we have  
\begin{align*}
\nonumber \|  \Upsilon'_n  (M)  \|  &=  \left\|  \begin{pmatrix}   a  I_n   & C^t  \\   C &  b  I_n  \end{pmatrix} \right\| %\\&
%\nonumber 
 \le  \left\|  \begin{pmatrix}   |a|    & \|C\|   \\    \|C\| & | b|    \end{pmatrix}  \right\| \\
\nonumber & 
 =  \frac{  |a| +  |b|   +  \sqrt{   (|b| -  |a|)^2   +  4  \|C\|^2} }2   \\
 &  \le  \frac{  |a| +  |b|   +  \sqrt{   (|b| -  |a|)^2   +  4   -  4|b|^2}}2  
\end{align*} 
Now, let us define $r:  =  |b|- |a|$, so that the bound becomes  
\begin{align*}
\|  \Upsilon'_n  (M)  \|    & \le  \frac{  2 |b| -r  +  \sqrt{   r^2   +  4   -  4|b|^2}}2  
\end{align*}
Note that, by definition, we have $  |b|  \ge r  \ge 0$.  
By maximizing over $|b|$, we obtain that the maximum is achieved for $|b|=  r$, yielding the bound  
\begin{align*}
\|  \Upsilon'_n  (M)  \|    & \le  \frac{   r  +  \sqrt{     4   -  3 r^2}}2    .
\end{align*}
The maximum over $r$ in the range $[0,1]$ is attained for $r=  1/\sqrt 3$, which yields the bound 
\begin{align*}
\|  \Upsilon'_n  (M)  \|    & \le  \frac{   1/\sqrt 3  +  \sqrt{   3}}2     =  \frac  { 4}{2\sqrt 3}  =  \frac{2}{\sqrt 3}    . \qedhere
\end{align*}
\end{proof}

%%%%%%%%%%%%%%%%%%%%%%%
\section{An unextendible positive isometry, complex case}  
%%%%%%%%%%%%%%%%%%%%%%%

In this section, we construct a positive map on a complex operator system inside $M_{2n}(\bb C)$ which is unital and isometric,
but still fails to have a positive extension.
The restriction to the corresponding real operator system will be a real linear map which is positive, unital and isometric.
Unlike the example in the previous section, this map will clearly have a positive, unital, isometric complexification.
Nevertheless, it does not have a positive extension to a map on $M_{2n}(\bb R)$.

We define an operator system
\begin{align*} %\label{Tn}
\spc T_n &:= \left\{  \begin{pmatrix}   A  & b I_n  \\  c I_n &  d  I_n  \end{pmatrix} :   A\in M_n(\C)  ,  b ,  c , d  \in   \C  \right\} 
%\\ \intertext{and}
%\spc T'_n &:= \left\{  \begin{pmatrix}   A  & b I_n  \\  c I_n &  d  I_n  \end{pmatrix} :   A\in M_n(\C)  ,  b ,  c , d  \in   \C  \right\}  
\end{align*} 
and a map $\Gamma_n:  \spc T_n\to \spc T_n$ by
\begin{align*}
 \Gamma_n  \begin{pmatrix}  A  & b I_n  \\  c  I_n &  d I_n  \end{pmatrix} 
  =  \begin{pmatrix}  A^t  &  b  I_n \\  c  I_n &  d I_n   \end{pmatrix}  .
\end{align*}
%Let $\Gamma_n = \Gamma_n'|_{\spc T_n}$ and observe that the range of $\Gamma_n$ is $\spc T_n$.
The map  $\Gamma_n$ is an  involution.   
%Observe that $\spc T_n'$ is the complexification of $\spc T_n$ and $\Gamma_n'$ is the complexification of $\Gamma$.

%%%%%%%%%%%%%%%%%%%%%%%
\begin{prop}
The map  $\Gamma_n$ ia positive.  
\end{prop}

\begin{proof}
By Lemma~\ref{lem:Snpositive}, the matrix $N =  \begin{pmatrix}   A  & b I_n  \\  c I_n &  d  I_n  \end{pmatrix}$ is positive 
if and only if $A\ge0$, $c = \bar b$, $d\ge0$ and $dA \ge |b|^2 I_n$.
Since $A$ and $A^t$ have the same spectrum,  Lemma~\ref{lem:Snpositive} shows that $\Gamma_n(N) \ge0$. 
%By restriction, {\color{blue} we obtain that $\Gamma_n$ is positive preserving.  Since $\Gamma_n$ is self-adjoint, we conclude that it is positive. }  {\color{red} $\Gamma_n$ is also positive.}
\end{proof}

Next we show that $\Gamma_n$ has unit norm.  
We require the following.

%%%%%%%%%%%%%%%%%%%%%%%
\begin{thm}\label{thm:bc} 
For every $A  \in M_n (\C)$,  $b,c,d \in  \C$,   the  matrices
\begin{align*}
 M  =  \begin{pmatrix}   A    &   b I_n  \\  c I_n  &  d  I_n    \end{pmatrix}  \quad \text{and}  \quad  
 N  =  \begin{pmatrix}   A    &   c  I_n  \\  b I_n &  d  I_n     \end{pmatrix} 
\end{align*}  
have the same singular values. 
\end{thm} 

\begin{proof}  
The singular values of the matrix $M  =  \begin{pmatrix}   A  &  b I_n  \\  c I_n  &  d I_n   \end{pmatrix}$ 
are the square roots of the eigenvalues of the matrix 
\begin{align*}
\nonumber M^* M  &  =   
 \begin{pmatrix}   A^*  & \overline c I_n  \\ \overline b  I_n &  \overline d  I_n  \end{pmatrix} 
 \begin{pmatrix}   A    &   b  I_n  \\  c  I_n  &  d   I_n  \end{pmatrix} \\&  
 =  \begin{pmatrix}   A^*  A  + |c|^2  I_n  & A^*  b  + \overline c   d  I_n  \\    A  \overline b +  c \overline d  I_n  &   ( |b|^2  +  |d|^2)   I_n  \end{pmatrix}  .
\end{align*}
In turn, the eigenvalues of $M^* M$ are the zeros of the characteristic polynomial 
\begin{align*}
\nonumber p_M(\lambda)  &  =  \det  (  M^* M  - \lambda   I_{2n})  \\
 \nonumber  &  = \det \begin{pmatrix}  A^* A + (|c|^2 -\lambda) I_n  &  A^*  b  +  \overline c  d  I_n  \\  A  \overline b +  c \overline d  I_n  
 & ( |b|^2  +  |d|^2  - \lambda)  I_n  \end{pmatrix} . 
\end{align*}
The determinant can be computed with the formula 
\[
 \det \begin{pmatrix}  A  & B  \\  C  &  D   \end{pmatrix}  =  \det (A  -  B D^{-1}  C)  \det  (D)  =  \det  ( A D  -   B D^{-1} C  D)
\]
provided that $D$ is invertible.    
In our case, $D$ commutes with $C$, and the formula simplifies to  
\[
 \det \begin{pmatrix}  A  &  B \\  C  &   D  \end{pmatrix}  =  \det  ( A  D   -   B  C) .
\]
Hence, we obtain 
\begin{align*}
\nonumber  p_M(\lambda) &  = \det   \Big\{   [A^*   A   + (|c|^2 -\lambda)  I_n ][ (|b|^2  +  |d|^2  - \lambda ) I_n ]  \\
\nonumber & \qquad  \qquad   -    ( A^*  b  +   \overline c   d    I_n   )     (A  \overline b  +   c  \overline d   I_n )  \Big\}  \\
  &  =  \det\Big\{  A^* A   ( |d|^2  - \lambda)    -  A   \overline b  \overline c   d   -  A^*   b  c \overline d   \\
\nonumber   &  \qquad\qquad  +     (   | cb |^2    -  |c|^2  \lambda  -  |b|^2  \lambda   -  |d|^2   \lambda  +  \lambda^2)    I_n    \Big\}   .
\end{align*}     
From the above expression, one can see that $p_M  (\lambda)  = p_N (\lambda)$.  
Since  the two polynomials  are identical, they have the same roots. 
Hence, $M^* M$ and $N^* N$ have the same eigenvalues, and  $M$ and $N$ have the same singular values.  
\end{proof}

%%%%%%%%%%%%%%%%%%%%%%%
\begin{cor}
The map $\Gamma_n$ is an isometry. 
\end{cor} 

\begin{proof} 
Let  $M  =   \begin{pmatrix}   A    &   b  I_n  \\   c  I_n  &  d   I_n   \end{pmatrix}$ be a generic element of $\spc T_n$.   
Theorem~\ref{thm:bc} implies that the norm of $M$ is equal to the norm of $N  = \begin{pmatrix}   A    &   c  I_n  \\   b  I_n  &  d   I_n   \end{pmatrix}$.  
In turn, the norm of $N$ is equal to the norm of $N^t   =  \begin{pmatrix}   A^t    &   b  I_n  \\   c I_n  &  d  I_n   \end{pmatrix} =  \Gamma_n (M)$.  
Hence, we have $\|  \Gamma_n  (  M) \|  =  \|  M\|$ for every $M \in \spc T_n'$. 
By restriction, $\Gamma_n$ is also an isometry.
\end{proof}

 %%%%%%%%%%%%%%%%%%%%%%%
\begin{theo}
The map $\Gamma_n$ is a positive isometry on $\cl T_n$ which has no positive extension to the full matrix algebra $M_{2n}  (\C)$. 
\end{theo}

\begin{proof}
The proof is by contradiction. 
We will show that, if  $\Gamma_n$ can be extended to a positive map  $\Psi_n$, then the map $\Psi_n$ must be
\begin{align*} %\label{partialtranspose}
 \Psi_n  \begin{pmatrix}   A    &   B \\  C   &   D  \end{pmatrix} 
 = \begin{pmatrix}   A^t    &   B^t \\  C^t   &   {D^t}   \end{pmatrix}. 
\end{align*} 
This is a contradiction, because this map is not positive.  

We start from the Kadison-Schwarz inequality for a positive  norm-one map    
\begin{align} \label{KS}
 \Psi_n (  M^2)  \ge   \big( \Psi_n (M) \big)^2    ,
\end{align} 
valid for an arbitrary self-adjoint matrix $M$ (see  \cite[Theorem 1.3.1]{stormer-book}). 
Let $M  =  \begin{pmatrix}   A  &  \overline c  I_n  \\ c I_n  &  d  I_n  \end{pmatrix}$, 
where $A^* =  A$ and $d \in \R$ be a generic self-adjoint element in $\spc T_n$.   
We have 
\begin{align*}
 M^2  & = \begin{pmatrix}   A^2 +  |c|^2 I_n &   \overline{c} (A  b  +  d  I_n)  \\   c( A + d  I_n) & ( |c|^2  +  |d|^2  )   I_n \end{pmatrix}  ,
\end{align*}
\begin{align*}
\nonumber \Psi_n (M^2 ) & 
=   \begin{pmatrix} 
  (A^t)^2  +  |c|^2 I_n  &  \overline c  d  I_n  \\  
  c  d  I_n & ( |b|^2  +  |d|^2  )   I_n  
  \end{pmatrix} 
  + \Psi_n  \begin{pmatrix}   0  &  \overline{c} A  \\   c A  &   0   \end{pmatrix} ,
\end{align*}
and  
\begin{align*}
 \big( \Psi_n (M) \big)^2 & =  
\begin{pmatrix}   
  (A^t)^2   +  |c|^2  I_n &   \overline{c} (A^t  b  +  d  I_n)  \\   
  c( A^t + d  I_n) & ( |c|^2  +  |d|^2  )   I_n    
\end{pmatrix} .
\end{align*}
Hence, Eq.~(\ref{KS})  implies 
\begin{align*}
  \Psi_n \begin{pmatrix}  0  &  \overline{c} A  \\   c A  &   0   \end{pmatrix}
  &\ge   \begin{pmatrix}  0  & \overline{c} A^t  \\   c A^t  &   0   \end{pmatrix} .
\end{align*}
If we replace $c$ with $-c$, we obtain 
\begin{align*}
 \Psi_n \begin{pmatrix}  0  &  -\overline{c} A  \\   -c A  &   0   \end{pmatrix}
  &\ge  \begin{pmatrix}  0  & -\overline{c} A^t  \\   -c A^t  &   0   \end{pmatrix} .
\end{align*}
for every self-adjoint matrix $A$.  Therefore
\begin{align*}
  \Psi_n \begin{pmatrix}  0  &  \overline{c} A  \\   c A  &   0   \end{pmatrix}
  &=   \begin{pmatrix}  0  & \overline{c} A^t  \\   c A^t  &   0   \end{pmatrix} .
\end{align*}
for every self-adjoint matrix $A$ and complex number $c$.  
Hence for an arbitrary element of $M_n(\C)$, written $C = A + iB$, where $A=A^*$ and $B=B^*$, 
\begin{align*} %\label{onsa}
 \Psi_n  \begin{pmatrix}  0  &  C^*  \\  C &   0   \end{pmatrix}   
  &=      \begin{pmatrix}   0  &  A^t - iB^t  \\  A^t + iB^t  &   0   \end{pmatrix}
  =      \begin{pmatrix}   0  &  \overline C  \\  C^t  &   0   \end{pmatrix}
  \qquad \forall C  \in M_n  (\C) .
\end{align*}
Therefore
\begin{align*}
\nonumber  \Psi_n  \begin{pmatrix}  0  &  0  \\  C  &   0   \end{pmatrix}   
  &=  \frac12 \Psi_n \begin{pmatrix}  0  &  C^*  \\  C  &   0   \end{pmatrix} 
  + \frac i2 \Psi_n \begin{pmatrix}  0  &  i C^*  \\  -iC &   0   \end{pmatrix} \\&
  = \frac12 \begin{pmatrix}   0  &  \overline C  \\  C^t  &   0   \end{pmatrix}
  + \frac i2 \begin{pmatrix}   0  &  i \overline C   \\  -i C^t &   0   \end{pmatrix}
  = \begin{pmatrix}   0  &  0   \\  C^t  &   0   \end{pmatrix} .
\end{align*}
Similarly the same holds for the $1,2$ entry. Hence
\begin{align*} %\label{onsa2}
 \Psi_n \begin{pmatrix}   0  &  B   \\ C  &   0   \end{pmatrix}   
 &=  \begin{pmatrix}    0    &    B^t   \\ C^t &   0  \end{pmatrix} 
 \qquad \forall B , C  \in M_n  (\C) ,
\end{align*} 

Now, there are two ways to conclude the proof. 
One way is to notice that the restriction of the map $\Psi_n$ to the real operator system $\spc S_n$ defined in Eq.~(\ref{Sn}) 
coincides with the map $\Upsilon_n$ defined in Eq.~(\ref{Phin}). 
Since we already proved that $\Upsilon_n$ cannot be extended to a positive map on the whole algebra (Theorem ~\ref{Phinunextendible}), 
we conclude that the map $\Psi_n$ cannot exist.  

Another, more explicit way is as follows.  
Since $\Gamma_n$ is unital, so must be $\Psi_n$.  
Since $\Psi_n$ is positive and unital, we must have 
\begin{align*} 
%\nonumber 
\Psi_n \begin{pmatrix}  0   &  0  \\  0   &  D   \end{pmatrix} 
\nonumber  &\le     \Psi_n  \begin{pmatrix}  0    &  0  \\ 0   &  I_n  \end{pmatrix}   %\\&
\nonumber  =    \Gamma_n'  \begin{pmatrix}  0    &  0  \\   0   &  I_n  \end{pmatrix}   %\\&
  =  \begin{pmatrix}  0    &  0  \\   0   &  I_n   \end{pmatrix} 
\end{align*}
for every $D \in M_n(\bb C)$ such that $0 \le D \le I_n$.   
Hence, there is a positive unital map $\chi$ on $M_n(\bb C)$ so that
\begin{align*}
   \Psi_n  \begin{pmatrix}  0    &  0  \\  0   &  D   \end{pmatrix}
   = \begin{pmatrix}  0    &  0  \\   0   &  \chi (D)   \end{pmatrix}  , 
\end{align*}
for $D \in M_n (\C)$.
If $0 \le D \le I_n$,  
\begin{align*}
 0 \le \Psi_n  \begin{pmatrix}    D    &    D   \\   D  &   D   \end{pmatrix} 
  =  \begin{pmatrix}    D^t    &    D^t   \\   D^t   &   \chi (D)   \end{pmatrix} .
\end{align*}
The matrix on the right-hand side is positive if and only if $\chi  (D)  \ge D^t$.    
Applying this also to $I_n-  D$, we get 
\[
 I_n - \chi(D) = \chi(I_n-D) \ge I_n - D^t .  
\]
Therefore $\chi(D) = D^t$. By linearity, this holds for all $D \in M_n (\bb C)$. 
Summarizing, we obtain
\begin{align*}
 \Psi_n \begin{pmatrix}   A   &   B  \\   C   &  D   \end{pmatrix} 
 =  \begin{pmatrix}  A^t  &  B^t  \\  C^t  &  D^t   \end{pmatrix}     . 
\end{align*}
The above map is known to be not positive. For example, one has 
\begin{align*}
\Psi_n    \left(   \begin{array}{cc|cc}    
1  &  0  & 0  & 1  \\  
0  &  0  & 0  & 0  \\ \hline
0  &  0  & 0  & 0  \\
1  &  0  & 0  & 1  
\end{array}
 \right)  = 
 \left( \begin{array}{cc|cc}    
1  &  0  & 0  & 0  \\  
0  &  0  & 1  & 0  \\ \hline
0  &  1  & 0  & 0  \\
0  &  0  & 0  & 1  
\end{array}
 \right)  . 
\end{align*}
The matrix on the left is positive, and the matrix on the right is not. 
\end{proof}  

Notice that the complex field plays an essential role in this last example.  While the map $\Gamma_n $ cannot be extended to the whole algebra $ M_{2n}  (\C)$,  the restriction of $\Gamma_n$ to the real operator system   \begin{align*} %\label{Tn}
\spc R_n &:= \left\{  \begin{pmatrix}   A  & b I_n  \\  c I_n &  d  I_n  \end{pmatrix} :   A\in M_n(\R)  ,  b ,  c , d  \in   \R  \right\}
\end{align*} 
admits a positive extension $\Psi_n$, defined on the whole matrix  algebra $M_{2n} (\R)$ as  
$\Psi_n  \begin{spmatrix}   A &  B \\   C  &  D  
\end{spmatrix}     =   \begin{spmatrix}   A^t &  B \\   C  &  D  
\end{spmatrix}$. 
%%%%%%%%%%%%%%%%%%%%%%%
\section*{Acknowledgments}  
%%%%%%%%%%%%%%%%%%%%%%%

GC  was supported by the Hong Kong Research Grant Council through grant 17300918  and through 
the Senior Research Fellowship Scheme via SRFS2021-7S02, and by the Croucher foundation. 
KRD and VIP were partially supported by the Natural Sciences and Engineering Research Council of Canada (NSERC).
MR is supported by the European Research Council (ERC Grant Agreement No. 851716).

%%%%%%%%%%%%%%%%%%%%%%%
\bibliographystyle{amsplain}

\end{document}